\documentclass[review, english]{elsarticle}

\usepackage{lineno}
\modulolinenumbers[5]

\journal{Statistics \& Probability Letters}

\biboptions{authoryear}

\usepackage[utf8]{inputenc}
\usepackage[T1]{fontenc}
\usepackage{lmodern}
\usepackage{babel}

\usepackage{mathrsfs}
\usepackage{amsmath}
\usepackage{amssymb} % loads package amsfont by default.
\usepackage{amsthm}

\usepackage{graphicx}
\usepackage[font = footnotesize, skip = 2pt]{caption}

\usepackage{booktabs}

\usepackage[unicode=true, colorlinks = true, linkcolor=blue, citecolor=blue]{hyperref}

\usepackage{mleftright} % corrects spacing after \left( and before \right), etc
\mleftright

\usepackage[section]{placeins} % prevents floats from violating a `section barrier'

\usepackage{accents} % For defining an underbar which displays better than underline

\usepackage[inline]{enumitem}

\usepackage{appendix}

\newtheorem{theorem}{Theorem}
\newtheorem{corollary}{Corollary}

\theoremstyle{remark}
  \newtheorem*{remark}{Remark}
\theoremstyle{definition}
\newcommand{\E}{\mathbb{E}}
\newcommand{\Prob}{\mathbb{P}}

\newcommand{\norm}[1]{\left\Vert{#1}\right\Vert}

\begin{document}

\begin{frontmatter}

\title{Identifying the Spectral Representation of Hilbertian Time Series}

\author[mymainaddress]{Eduardo Horta\corref{mycorrespondingauthor}}
\cortext[mycorrespondingauthor]{Corresponding author}
\ead{eduardo.horta@ufrgs.br}

\author[mymainaddress]{Flavio Ziegelmann}

\address[mymainaddress]{Universidade Federal do Rio Grande do Sul, Department of Statistics, 9500 Bento Gonçalves Av., 43--111, Porto Alegre, RS, Brazil 91509-900.}

\begin{abstract}
We provide $\sqrt{n}$-consistency results regarding estimation of the spectral representation of covariance operators of Hilbertian time series, in a setting with imperfect measurements. This is a generalization of the method developed in \cite{bathia2010identifying}. The generalization relies on an important property of centered random elements in a separable Hilbert space, namely, that they lie almost surely in the closed linear span of the associated covariance operator. We provide a straightforward proof to this fact. This result is, to our knowledge, overlooked in the literature. It incidentally gives a rigorous formulation of \textsc{pca} in Hilbert spaces.
\end{abstract}

\begin{keyword}
covariance operator\sep dimension reduction\sep Hilbertian time series\sep $\sqrt{n}$-consistency\sep functional \textsc{pca}
\MSC[2010] 60G10\sep 62G99\sep 62M99
\end{keyword}

\end{frontmatter}

%\linenumbers

\section{Introduction}
In this paper, we provide theoretical results regarding estimation of the spectral representation of the covariance operator of stationary Hilbertian time series. This is a generalization of the method developed in \cite{bathia2010identifying} to a setting of random elements in a separable Hilbert space. The approach taken in \cite{bathia2010identifying} relates to functional \textsc{pca} and, similarly to the latter, relies strongly on the Karhunen-Loève (K-L) Theorem. The authors develop the theory in the context of curve time series, with each random curve in the sequence satisfying the conditions of the K-L Theorem which, together with a stationarity assumption, ensures that the curves can all be expanded in the same basis -- namely, the basis induced by their zero-lag covariance function. The idea is to identify the dimension of the space $M$ spanned by this basis (finite by assumption), and to estimate $M$, when the curves are observed with some degree of error. Specifically, it is assumed that the statistician can only observe the curve time series $\left(Y_t\right)$, where
\begin{equation*}
	Y_t = X_t + \epsilon_t,
\end{equation*}
whereas the curve time series of interest is actually $\left(X_t\right)$. Here $Y_t$, $X_t$ and $\epsilon_t$ are random functions (curves) defined on $\left[0,1\right]$. Estimation of $M$ in this framework was previously addressed in \cite{hall2006assessing} assuming the curves are iid (in $t$), a setting in which the problem is indeed unsolvable in the sense that one cannot separate $X_t$ from $\epsilon_t$. \cite{hall2006assessing} propose a Deus ex machina solution which consists in assuming that $\epsilon_t$ goes to $0$ as the sample size grows. \cite{bathia2010identifying} in turn resolve this issue by imposing a dependence structure in the evolution of $\left(X_t\right)$. Their key assumption is that, at some lag $k$, the $k$-th lag autocovariance matrix of the random vector composed by the Fourier coefficients of $X_t$ in $M$, is full rank. In our setting this corresponds to Assumption~\ref{thm:bathia-thm-1-A1} (see below).

In \cite{bathia2010identifying} it is assumed that each of the stochastic processes $\left(X_t\left(u\right):\,u\in\left[0,1\right]\right)$ satisfy the conditions of the K-L Theorem (and similarly for $\epsilon_t$), and as a consequence the curves are in fact random elements with values in the Hilbert space $L^2\left[0,1\right]$. Therefore, since every separable Hilbert space is isomorphic to  $L^2\left[0,1\right]$, the idea of a generalization to separable Hilbert spaces of the aforementioned methodology might seem, at first, rather dull. The issue is that \emph{in applications transforming the data (that is, applying the isomorphism) may not be feasible nor desirable}. For instance, the isomorphism may involve calculating the Fourier coefficients in some `rule-of-thumb' basis which might yield infinite series even when the curves are actually finite dimensional.

The approach that we take here relies instead on the key feature that a centered Hilbertian random element of strong second order, lies almost surely in the closed linear span of its corresponding covariance operator. This result allows one to dispense with considerations of `sample path properties' of a random curve by addressing the spectral representation of a Hilbertian random element directly. In other words, the Karhunen-Loève Theorem is just a special case\footnote{This is not entirely true since the Karhunen-Loève Theorem states uniform (in $\left[0,1\right]$) $L^2\left(\Omega\right)$ convergence.} of a more general phenomena. The result below (which motivates -- and for that matter, justifies -- our approach) is not a new one: it appears, for example, in a slightly different guise as an exercise in \cite{vakhania1987probability}. However, it is in our opinion rather overlooked in the literature. The proof that we give is straightforward and, to our knowledge, a new one. In this paper $H$ is always assumed to be a real Hilbert space, but with minor adaptations all stated results hold for complex $H$.

\begin{theorem}\label{thm:xi-orthogonal-ker}
Let $H$ be a separable Hilbert space, and assume $\xi$ is a centered random element in $H$ of strong second order, with covariance operator $R$. Then $\xi \perp \ker(R)$ almost surely.
\end{theorem}

\begin{corollary}\label{thm:hilbert-representations}
In the conditions of Theorem~\ref{thm:xi-orthogonal-ker}, let $\left(\lambda_j:\,j\in J\right)$ be the (possibly finite) non-increasing sequence of nonzero eigenvalues of $R$, repeated according to multiplicity, and let $\left\{\varphi_j:\,j\in J\right\}$ denote the orthonormal set of associated eigenvectors. Then
\begin{enumerate}[label={\textit{(\roman*)}}, ref={\textit{(\roman*)}}]
\item \label{theorem-lemma-hilbert-representations-item-1}$\xi(\omega) = \sum_{j\in J}\langle \xi(\omega),\varphi_j\rangle\varphi_j$ in $H$, almost surely;
\item \label{theorem-lemma-hilbert-representations-item-2}$\xi = \sum_{j\in J} \langle \xi,\varphi_j\rangle \varphi_j$ in $L^2_\Prob(H)$.
\end{enumerate}
Moreover, the scalar random variables $\langle\xi,\varphi_i\rangle$ and $\langle\xi,\varphi_j\rangle$ are uncorrelated if $i\neq j$, with $\E\langle\xi,\varphi_j\rangle^2 = \lambda_j$.
\end{corollary}

\begin{remark}
\begin{enumerate*}[label={\textit{(\alph*)}}]
	\item Although it is beyond the scope of this work, we call attention to the fact that Theorem~\ref{thm:xi-orthogonal-ker} and Corollary~\ref{thm:hilbert-representations} provide a rigorous justification of \textsc{pca} for Hilbertian random elements.
	\item In Corollary~\ref{thm:hilbert-representations} either $J = \mathbb{N}$ or, whenever $R$ is of rank $d<\infty$, $J = \left\{1,\dots, d\right\}$.
\end{enumerate*}
\end{remark}

Proofs to the above and subsequent statements are given in Appendix~\ref{sec:proof}. We can now adapt the methodology of \cite{bathia2010identifying} to a more general setting.

\section{The model}

In what follows $(\Omega,\mathscr{F},\Prob)$ is a fixed complete probability space. Consider a stationary process $\left(\xi_t:\,t\in\mathbb{T}\right)$ of random elements with values in a separable Hilbert space $H$. Here $\mathbb{T}$ is either $\mathbb{N}\cup\left\{0\right\}$ or $\mathbb{Z}$. We assume throughout that $\xi_0$ is a centered random element in $H$ of strong second order. Of course, the stationarity assumption ensures that these properties are shared by all the $\xi_t$. Now let
\begin{equation*}
	R_k\left(h\right) := \E \langle \xi_0, h\rangle \xi_k, \qquad h\in H,
\end{equation*}
denote the $k$-th lag autocovariance operator of $\left(\xi_t\right)$, and let $\left(\lambda_j:\,j\in J\right)$ be the (possibly finite) non-increasing sequence of nonzero eigenvalues of $R_0$, repeated according to multiplicity. Here either $J = \mathbb{N}$ or, whenever $R_0$ is of rank $d<\infty$, $J = \left\{1,\dots, d\right\}$. Now for $j\in J$, let $\varphi_j\in H$ be defined by
\begin{equation*}
	R_0\left(\varphi_j\right) = \lambda_j\varphi_j,
\end{equation*}
and assume the set $\left\{\varphi_j: j\in J\right\}$ is orthonormal in $H$. Corollary~\ref{thm:hilbert-representations} and the stationarity assumption ensure that the spectral representation
\begin{equation*}
	\xi_t = \sum_{j\in J} Z_{tj} \varphi_j
\end{equation*}
holds almost surely in $H$, for all $t$, where the $Z_{tj}:=\langle \xi_t,\varphi_j\rangle$ are centered scalar random variables satisfying $\E Z_{tj}^2 = \lambda_j$ for all $t$, and $\E Z_{ti}Z_{tj} = 0$ if $i\neq j$. In applications, an important case is that in which the above sum has only finitely many terms: that is, the case in which $R_0$ is a finite rank operator. In this setting, the stochastic evolution of $\left(\xi_t\right)$ is driven by a vector process $\left(\boldsymbol{Z}_t:\,t\in\mathbb{T}\right)$, where $\boldsymbol{Z}_t = \left(Z_{t1},\dots,Z_{td}\right)$, in $\mathbb{R}^d$ (here $d$ is the rank of $R_0$). The condition that $R_0$ is of finite rank models the situation where the data lie (in principle) in an infinite dimensional space, but it is reasonable to assume that they in fact lie in a finite dimensional subspace which must be identified inferentially.

We are interested in modeling the situation where the statistician observes a process $\left(\zeta_t:\,t\in\mathbb{T}\right)$ of $H$-valued random elements, and we shall consider two settings; the simplest one occurs when
\begin{equation}\label{equation-perfect-measurement}
\zeta_t = \xi_t. 
\end{equation}
This is to be interpreted as meaning that perfect measurements of a `quantity of interest' $\xi_t$ are attainable. A more realistic scenario would admit that associated to every measurement there is an intrinsic error -- due to rounding, imprecise instruments, etc. In that case observations would be of the form
\begin{equation}\label{eq:noisy-measurement}
	\zeta_t = \xi_t + \epsilon_t.
\end{equation}
In fact, the latter model nests the `no noise' one if we allow the $\epsilon_t$ to be degenerate. Equation~\eqref{eq:noisy-measurement} is analogous to the model considered in \cite{hall2006assessing} and in \cite{bathia2010identifying}. Here $\left(\epsilon_t:\,t\in\mathbb{T}\right)$ is assumed to be \emph{noise}, in the following sense:
\begin{enumerate*}[label={\textit{(\roman*)}}, ref = {\textit{(\roman*)}}]
	\item for all $t$, $\epsilon_t\in L^2_\Prob\left(H\right)$, with $\E\epsilon_t = 0$;
	\item for each $t\neq s$, $\epsilon_t$ and $\epsilon_s$ are strongly orthogonal.
\end{enumerate*}

In the above setting, for $h,f\in H$ one has $\E\langle h, \zeta_t\rangle\langle f, \zeta_t\rangle = \langle R_0\left(h\right), f\rangle + \E\langle h,\epsilon_t\rangle\langle f,\epsilon_t\rangle$ and thus estimation of $R_0$ via a sample $\left(\zeta_1,\dots,\zeta_n\right)$ is spoiled (unless the $\epsilon_t$ are degenerate). This undesirable property has been addressed by \cite{hall2006assessing} and \cite{bathia2010identifying} respectively in the iid scenario and in the time series (with dependence) setting. The clever approach by \cite{bathia2010identifying} relies on the fact that $\E\langle h, \zeta_t\rangle\langle f, \zeta_{t+1}\rangle = \langle R_1\left(h\right),f\rangle$ (lagging filters the noise) and therefore $R_1$ can be estimated using the data $\left(\zeta_1,\dots,\zeta_n\right)$. Now an easy check shows that $\overline{\mbox{ran}\left(R_1\right)} \subset \overline{\mbox{ran}\left(R_0\right)}$. The key assumption in \cite{bathia2010identifying} is asking that this relation hold with equality:
\begin{enumerate}[label={\textup{(A\arabic*)}} ,ref={\textup{(A\arabic*)}}, start = 1]
\item \label{thm:bathia-thm-1-A1} $\overline{\mbox{ran}\left(R_1\right)} = \overline{\mbox{ran}\left(R_0\right)}$.
\end{enumerate}

Consider the operator $S := R_1^{} R_1^*$, where $*$ denotes adjoining. It is certainly positive, and compact (indeed nuclear) since $\overline{\mbox{ran}\left(R_1^{} R_1^*\right)} = \overline{\mbox{ran}\left(R_1\right)}$. Let $\left(\theta_j:\,j\in J'\right)$ be the (possibly finite) non-increasing sequence of nonzero eigenvalues of $S$, repeated according to multiplicity, and denote by $\left\{\psi_j:\,j\in J'\right\}$ the orthonormal set of associated eigenvectors. Under Assumption~\ref{thm:bathia-thm-1-A1} we have $J'=J$, and the representation
\begin{equation*}
	\xi_t = \sum_{j\in J} W_{tj}\psi_j
\end{equation*}
is seen to hold, for all $t$, almost surely in $H$ for centered scalar random variables $W_{tj} = \langle \xi_t, \psi_j\rangle$. Again, when $R_0$ is finite rank, say $\mbox{rank}\left(R_0\right)=d$, then the stochastic evolution of $\xi_t$ is driven by the finite-dimensional vector process $\left(\boldsymbol{W}_t:\,t\in\mathbb{T}\right)$, where $\boldsymbol{W}_t = \left(W_{t1},\dots,W_{td}\right)$.

\section{Main results}
Before stating our result, let us establish some notation. Define the estimator $\widehat{S} := \widehat{R}_1^{}\widehat{R}_1^*$, where $\widehat{R}_1$ is given by
\begin{equation*}
	\widehat{R}_1\left(h\right) := \frac{1}{n-1}\sum_{t=1}^{n-1} \langle \zeta_t, h\rangle \zeta_{t+1},\qquad h\in H.
\end{equation*}
Notice that $\widehat{R}_1$ is almost surely a finite rank operator, say of rank $q$, with $q\leq n-1$ almost surely, and thus $\widehat{S}$ is also of finite rank $q$. Let $\big(\widehat{\theta}_1,\widehat{\theta}_2,\dots\big)$ denote the non-increasing sequence of eigenvalues of $\widehat{S}$, repeated according to multiplicity. Clearly $\widehat{\theta}_j = 0$ if $j > n-1$. Denote by $\big\{\widehat{\psi}_1,\widehat{\psi}_2,\dots\big\}$ the orthonormal basis of associated eigenfunctions. Also, for a closed subspace $V\subset H$, let $\Pi_V$ denote the orthogonal projector onto $V$. Let $M := \overline{\mbox{ran}\left(R_0\right)}$, and for conformable $k$ put $\widehat{M}_k := \vee_{j=1}^k \widehat{\psi}_j$.

\begin{theorem}\label{thm:bathia-thm-1}
Let \ref{thm:bathia-thm-1-A1} and the following conditions hold.
\begin{enumerate}[label={\textup{(A\arabic*)}} ,ref={\textup{(A\arabic*)}}, start = 2]
\item \label{thm:bathia-thm-1-A2} $\left(\zeta_t:\,t\in\mathbb{T}\right)$ is strictly stationary and $\psi$-mixing, with the mixing coefficient satisfying the condition $\sum_{k=1}^\infty k\,\psi^{1/2}\left(k\right) < \infty$;
\item \label{thm:bathia-thm-1-A3} $\zeta_t\in L^4_\Prob\left(H\right)$, for all $t$;
\item \label{thm:bathia-thm-1-A4} $\epsilon_t$ and $\xi_s$ are strongly orthogonal, for all $t$ and $s$.
\end{enumerate}
Then,
\begin{enumerate}[label={\textit{(\roman*)}}, ref = {\textit{(\roman*)}}]
	\item \label{thm:bathia-thm-1-itm1} $\big\Vert \widehat{S} - S \big\Vert_2 = O_\Prob\left(n^{-1/2}\right)$;
	\item \label{thm:bathia-thm-1-itm3} $\sup_{j\in J}\big\vert \widehat{\theta}_j - \theta_j \big\vert = O_\Prob\left(n^{-1/2}\right)$.
\end{enumerate}
Moreover, if
\begin{enumerate}[label={\textup{(A\arabic*)}} ,ref={\textup{(A\arabic*)}}, start = 5]
\item \label{thm:bathia-thm-1-A5} $\ker\left(S - \theta_j\right)$ is one-dimensional, for each nonzero eigenvalue $\theta_j$ of $S$,
\end{enumerate}
holds, then 
\begin{enumerate}[label={\textit{(\roman*)}}, ref = {\textit{(\roman*)}}, start = 3]
\item \label{thm:bathia-thm-1-itm2} $\sup_{j\in J}\big\Vert \widehat{\psi}_j - \psi_j \big\Vert = O_\Prob\left(n^{-1/2}\right)$.
\end{enumerate}
If additionally $S$ is of rank $d < \infty$, then
\begin{enumerate}[label={\textit{(\roman*)}}, ref = {\textit{(\roman*)}}, start = 4]
	\item \label{thm:bathia-thm-1-itm4} $\widehat{\theta}_j = O_\Prob\left(n^{-1}\right)$, for all $j > d$;
	\item \label{thm:bathia-thm-1-itm5} $\big\Vert \Pi_{M}\big(\widehat{\psi}_j\big) \big\Vert = O_\Prob\left(n^{-1/2}\right)$, for all $j > d$.
\end{enumerate}
\end{theorem}

\begin{remark}
\begin{enumerate*}[label={\textit{(\alph*)}}]
\item Assumption~\ref{thm:bathia-thm-1-A5} ensures that $\psi_j$ is an identifiable statistical parameter. It is assumed that the `correct' version (among $\psi_j$ and $-\psi_j$) is being picked. See Lemma 4.3 in \cite{bosq2000linear};
\item Since the operator $\widehat{S}$ is almost surely of finite rank, items \ref{thm:bathia-thm-1-itm3} and \ref{thm:bathia-thm-1-itm4} imply the following. If $\mathrm{rank}(S)=d<\infty$, then for $j=1,\dots,d$, $\widehat{\theta}_j$ is eventually non-zero and arbitrarily close to $\theta_j$, and the remaining nonzero $\widehat{\theta}_j$ for $j>d$ (if any) are eventually arbitrarily close to zero. Otherwise, eventually $\widehat{\theta}_j>0$ for all $j$ (but notice that this cannot occur uniformly in $j$: it is always the case that $\widehat{\theta}_j=0$ for $j>n-1$). This property can be used to propose consistent estimators of $d$.
\end{enumerate*}
\end{remark}

\begin{corollary}\label{thm:bathia-corollary}
Let Assumptions~\ref{thm:bathia-thm-1-A1}--\ref{thm:bathia-thm-1-A4} hold. Let $N_j := \ker\left(S - \theta_j\right)$ and $\widehat{N}_j := \ker\big(\widehat{S} - \widehat{\theta}_j\big)$. Then,
\begin{enumerate}[label={\textit{(\roman*)}}, ref={\textit{(\roman*)}}]
\item \label{thm:bathia-corollary-itm1} $\norm{\Pi_{\widehat{N}_j} - \Pi_{\vphantom{\widehat{N}_j}N_j}}_2 = O_\Prob\left(n^{-1/2}\right)$, for all $j$ such that $N_j$ is one-dimensional;
\item \label{thm:bathia-corollary-itm2} if $S$ is of rank $d < \infty$, then $\norm{\Pi_{\widehat{M}_d} - \Pi_{{\vphantom{\widehat{M}_d}}M}}_2 = O_\Prob\left(n^{-1/2}\right)$;
\item \label{thm:bathia-corollary-itm3} if $S$ is of rank $d < \infty$, there exists a metric $\rho$ on the collection of finite-dimensional subspaces of $H$ such that $\rho\big(\widehat{M}_d, M\big) = O_\Prob\left(n^{-1/2}\right)$.
\end{enumerate}
\end{corollary}

\begin{remark}
\begin{enumerate*}[label={\textit{(\alph*)}}]
\item Observe that, when the process $(\xi_t)$ is not centered, evidently all the above results would still hold by replacing $\zeta_t$ by $\zeta_t-\E\xi_0$ and $\xi_t$ by $\xi_t-\E\xi_0$, but this is not practical since in general $\E\xi_0$ is not known to the statistician. However, this does not pose a problem, since under mild conditions we have $1/n\sum_{t=1}^n\zeta_t\overset{a.s}\rightarrow \E\xi_0$, and thus all the results still hold with $\zeta_t$ and $\xi_t$ replaced respectively by $\zeta_t - 1/n\sum_{t=1}^n\zeta_t$ and $\xi_t - 1/n\sum_{t=1}^n\zeta_t$;
\item The key assumption in \cite{bathia2010identifying} would be translated in our setting to the condition that, for some $k\geq 1$, the identity $\overline{\mbox{ran}\left(R_{k}\right)} = \overline{\mbox{ran}\left(R_0\right)}$ holds. For simplicity we have assumed that $k=1$, but of course the stated results remain true if we take $k$ to be any integer $\geq 1$ and redefine $S$ and $\widehat{S}$ appropriately. Indeed the stated results remain true if we define $S = \left(n-p\right)^{-1}\sum_{k=1}^p R_k^{} R_k^*$, where $p$ is an integer such that $\overline{\mbox{ran}\left(R_{k}\right)} = \overline{\mbox{ran}\left(R_0\right)}$ holds for some $k\leq p$. In statistical applications, a recommended approach would be to estimate $S$ defined in this manner. In any case, computation of the eigenvalues and eigenvectors of $\widehat{S}$ can be carried out directly through the spectral decomposition of a convenient $n-p \times n-p$ matrix. The method is discussed in \cite{bathia2010identifying}. Notice that if $R_0$ is of rank one, then asking that $\overline{\mbox{ran}\left(R_{k}\right)} = \overline{\mbox{ran}\left(R_0\right)}$ holds for some $k$ corresponds to the requirement that the times series $\left(Z_{t1}:\,t\in\mathbb{T}\right)$ is correlated at some lag $k$. Otherwise we would find ourselves in the not very interesting scenario (for our purposes) of an uncorrelated time series.
\end{enumerate*}
\end{remark}

\section{Concluding remarks}
In this paper we have provided consistency results regarding estimation of the spectral representation of Hilbertian time series, in a setting with imperfect measurements. This generalizes a result from \cite{bathia2010identifying}. The generalization relies on an important property of centered random elements in a separable Hilbert space -- see Theorem~\ref{thm:xi-orthogonal-ker}. Further work should be directed at obtaining a Central Limit Theorem for the operator $\widehat{S}$, which would have the important consequence of providing Central Limit Theorems for its eigenvalues (via Theorem 1.2 in \cite{mas2003perturbation}), potentially allowing one to propose statistical tests for these parameters. The term `spectral' in the title of this work refers, of course, to the spectral representation of the operator $S$ and not to the spectral representation of the time series $\left(\xi_t\right)$ in the usual sense.

\begin{appendices}

\numberwithin{equation}{section}

\section{Notation and mathematical background}
As in the main text we let $(\Omega,\mathscr{F},\Prob)$ denote a complete probability space, i.e. a probability space with the additional requirement that subsets $N\subset\Omega$ with outer probability zero are elements of $\mathscr{F}$. Let $H$ be a separable Hilbert space with inner-product $\langle\cdot,\cdot\rangle$ and norm $\Vert\cdot\Vert$. A Borel measurable\footnote{There are notions of strong and weak measurability but for separable spaces they coincide.} map $\xi:\Omega\rightarrow H$ is called a \emph{random element with values in $H$} (also: Hilbertian random element). For $q\geq 1$, if $\E \Vert{\xi}\Vert^q < \infty$ we say that $\xi$ is of \emph{strong order $q$} and write $\xi\in L^q_\Prob\left(H\right)$. In this case, there is a unique element $h_\xi\in H$ satisfying the identity $\E\langle \xi, f\rangle = \langle h_\xi, f\rangle$ for all $f \in H$. The element $h_\xi$ is called the \emph{expectation} of $\xi$ and is denoted be $\E\xi$. If $\E\xi = 0$ we say that $\xi$ is \emph{centered}. If $\xi$ and $\eta$ are centered random elements in $H$ of strong order $2$, they are said to be (mutually) strongly orthogonal if, for each $h, f\in H$, it holds that $\E\langle h, \xi\rangle\langle f, \eta\rangle = 0 $.

Denote by $\mathcal{L}\left(H\right)$ the Banach space of bounded linear operators acting on $H$. Let $A\in\mathcal{L}\left(H\right)$. If for some (and hence, all) orthonormal basis $\left(e_j\right)$ of $H$ one has $\norm{A}_2 := \sum_{j=1}^\infty \norm{A\left(e_j\right)}^2 < \infty$, we say that $A$ is a \emph{Hilbert-Schmidt operator}. The set $\mathcal{L}_2\left(H\right)$ of Hilbert-Schmidt operators is itself a separable Hilbert space with inner-product $\langle A, B\rangle_2 = \sum_{j=1}^\infty \langle A\left(e_j\right), B\left(e_j\right)\rangle$, with $\norm{\cdot}_2$ being the induced norm. An operator $T \in \mathcal{L}\left(H\right)$ is said to be \emph{nuclear}, or \emph{trace-class}, if $T = AB$ for some Hilbert-Schmidt operators $A$ and $B$. If $\xi\in L^2_\Prob\left(H\right)$, its \emph{covariance operator} is the nuclear operator $R_\xi\left(h\right) := \E\langle\xi,h\rangle\xi$, $h\in H$. More generally, if $\xi, \eta\in L^2_\Prob\left(H\right)$, their \emph{cross-covariance} operator is defined, for $h\in H$, by $R_{\xi,\eta}\left(h\right) := \E\langle\xi, h\rangle\eta$. In the main text we denote by $R_k$ the cross-covariance operator of $\xi_0$ and $\xi_k$.

For a survey on strong mixing processes, including the definition of $\psi$-mixing in Assumption~\ref{thm:bathia-thm-1-A2}, we refer the reader to \cite{bradley2005basic}.

\section{Proofs}\label{sec:proof}
\begin{proof}[Proof of Theorem~\ref{thm:xi-orthogonal-ker}]
Let $(e_j)$ be a basis of $\ker(R)$. It suffices to show that $\E \left\vert \langle \xi, e_j\rangle\right\vert^2=0$ for each $j$. Indeed, this implies that there exist sets $E_j$, $\Prob(E_j)=0$ and $\langle \xi(\omega), e_j\rangle=0$ for $\omega \notin E_j$. Thus $\langle \xi(\omega),e_j\rangle=0$ for all $j$ as long as $\omega\notin \bigcap E_j$ with $\Prob\left(\bigcap E_j\right)=0$. But $\E \left\vert \langle \xi, e_j\rangle\right\vert^2 = \E \langle  \xi, e_j\rangle \langle \xi, e_j\rangle=\E \langle {\langle  \xi, e_j\rangle}\xi, e_j\rangle={\langle \E\langle \xi, e_j\rangle\xi, e_j\rangle}=\langle R(e_j),e_j\rangle=0$.
\end{proof}

\begin{proof}[Proof of Corollary~\ref{thm:hilbert-representations}] Item~\ref{theorem-lemma-hilbert-representations-item-1} is just another way of stating the Lemma. For item~\ref{theorem-lemma-hilbert-representations-item-2}, first notice that the functions $\omega\mapsto \langle \xi(\omega),\varphi_j\rangle\varphi_j$, $j\in J$, form an orthogonal set in $L^2_\Prob(H)$ (although not \emph{orthonormal}). We must show that $\int \Vert \xi(\omega) - \sum_{j=1}^n\langle\xi(\omega),\varphi_j\rangle\varphi_j\Vert^2 d\Prob(\omega) \rightarrow 0$. Let $g_n(\omega):= \Vert \xi(\omega) - \sum_{j=1}^n\langle\xi(\omega),\varphi_j\rangle\varphi_j\Vert$. By item~\ref{theorem-lemma-hilbert-representations-item-1} $g_n(\omega)\rightarrow 0$ almost surely. Also, $0\leq g_n(\omega)\leq 2 \left\Vert \xi(\omega)\right\Vert$. So $g_n^2(\omega)\rightarrow 0$ and $g_n^2(\omega)\leq 4\left\Vert\xi(\omega)\right\Vert^2$. Now apply Lebesgue's Dominated Convergence Theorem.
\end{proof}

\begin{proof}[Proof of Theorem~\ref{thm:bathia-thm-1}]
One only has to consider an isomorphism $U:H\rightarrow L^2\left[0,1\right]$. The proof is the same as in \cite{bathia2010identifying}.
\end{proof}

\begin{proof}[Proof of Corollary~\ref{thm:bathia-corollary}]
See the proof of Theorem~2 in \cite{bathia2010identifying}.
\end{proof}

\begin{remark}
The hypothesis that $\xi$ is centered in Theorem~\ref{thm:xi-orthogonal-ker} cannot be weakened, as the following simple example shows. Let $H=\mathbb{R}^2$ and let $\xi=(\xi_1,\xi_2)$ where $\xi_1$ is a (real valued) standard normal and $\xi_2=1$ almost surely. Then $R\equiv (R_{ij})$ is the matrix with all entries equal to zero except for $R_{11}$ which is equal to $1$, and obviously one has $\Prob(\xi\perp\ker(R))=0$.
\end{remark}

\end{appendices}

%\section*{References}

\end{document}